\documentclass{amsart}
\usepackage[T1]{fontenc} 

\usepackage{amssymb}
\usepackage{amsthm}
\usepackage{amsmath}
\usepackage{xcolor}
\usepackage{amsmath}
\usepackage{amsfonts}
\usepackage[utf8]{inputenc}
\usepackage[numbers]{natbib}
\usepackage{graphicx}
\usepackage{mathtools}
\usepackage{amssymb}
\usepackage{amsthm}
\usepackage{tikz-cd} 
\usepackage{enumerate}
\usepackage[english]{babel}
\usepackage{hyperref}
\usepackage{tikz}
\usetikzlibrary{positioning} 
\usepackage{nomencl}
\usepackage[matrix, arrow]{xy}
\xyoption{arrow}
\theoremstyle{plain}\newtheorem{Theorem}{Theorem}[section]
\theoremstyle{plain}\newtheorem{Conjecture}[Theorem]{Conjecture}
\theoremstyle{plain}
\theoremstyle{plain}
\theoremstyle{plain}\newtheorem{Proposition}[Theorem]{Proposition}
\theoremstyle{definition}
\theoremstyle{definition}
\theoremstyle{definition}
\theoremstyle{definition}
\theoremstyle{definition}
\theoremstyle{definition}
\theoremstyle{definition}

\def\sp{\mathrm{Sp}}

\def\Syl{\mathrm{Syl}}

\DeclareMathOperator{\su}{SU}

\DeclareMathOperator{\slg}{SL}

\DeclareMathOperator{\gu}{GU}

\DeclareMathOperator{\mor}{Mor}

\DeclareMathOperator{\syl}{Syl}
\DeclareMathOperator{\gl}{GL}

\makeindex
\title{Exotic Fusion System on a Subgroup of the Monster} 
\author{Patrick Serwene}
\address{Technische Universität Dresden, Faculty of Mathematics, 01062 Dresden, Germany}
\email{patrick.serwene@tu-dresden.de}

\date{\today}

\begin{document}

\begin{abstract}
We prove that an exotic fusion system described by Grazian on a subgroup of the Monster group is block-exotic, thus proving that exotic and block-exotic fusion systems are the same for all $p$-groups with sectional rank $3$, where $p \geq 5$. 
\end{abstract}

\maketitle
\section{Introduction}
Consider a prime number $p$ and a finite $p$-group $P$. We define a \textit{fusion system} as a category where the objects are the subgroups of $P$, referred to as a category on $P$, and the morphisms are injective group homomorphisms between these subgroups, subject to specific conditions. If these morphisms adhere to two additional axioms, we term the fusion system \textit{saturated}. For simplicity, we shall henceforth use the term "fusion system" to mean "saturated fusion system."\\
Every finite group $G$ gives rise to a fusion system $\mathcal F_P(G)$ on a Sylow $p$-subgroup $P$ of $G$, where the morphisms are defined as conjugation maps induced by a fixed element in $G$, where they are well-defined. A fusion system constructed in this manner is termed \textit{realizable}, while one that cannot be constructed as such is termed \textit{exotic}. Let $k$ be an algebraically closed field with characteristic $p$, and let $b$ be a block of $kG$. In this context, we can also define a fusion system on a defect group $P$ of $b$ by defining the morphisms as well-defined conjugation maps induced by an element in $G$. This fusion system is denoted by $\mathcal F_{(P,e_P)}(G,b)$, where $(P,e_P)$ is a maximal $b$-Brauer pair. Not every fusion system $\mathcal F$ can be realized in this way; if it can, we call $\mathcal F$ \textit{block-realizable}; otherwise, it is termed \textit{block-exotic}. The following fact follows from Brauer’s Third Main Theorem (see \cite[Theorem 3.6]{radha}): If $G$ is a finite group and $b$ is the principal $p$-block of $kG$, i.e., the block corresponding to the trivial character, with maximal $b$-Brauer pair $(P,e_P)$, then $P \in \syl_p(G)$ and $\mathcal{F}_{(P,e_P)}(G,b)=\mathcal{F}_P(G)$. Hence, any realizable fusion system is block-realizable. However, the converse remains an open problem and has been conjectured for some time, see \cite[Part IV,7.1]{ako} and \cite[9.4]{david}:

\begin{Conjecture}
\label{XX}
\textit{If $\mathcal F$ is an exotic fusion system, then $\mathcal F$ is block-exotic.}
\end{Conjecture}

This conjecture is hard to tackle, since so far it has been mainly proved for one exotic family of fusion systems at a time, see \cite{solom}, \cite{david}, \cite{ks}, \cite{S20} or \cite{S24}. Since block fusion systems "misbehave" with regards to normal subgroups, a descend to normal subgroups and thus a reduction is not easily done and requires more general structures, see \cite{ks} or \cite{ps23} for an explanation.\\
In this paper, we prove Conjecture \ref{XX} for all fusion systems on $p$-groups with sectional rank $3$, where $p \geq 5$. A group has sectional rank $r$ if $r$ is the minimal integer such that any of its subgroups has at most rank $r$. In \cite{grazian}, fusion systems on $p$-groups of sectional rank $3$ are studied and if we further assume that $p \geq 5$ and $O_p(\mathcal F)=1$, Grazian proves in Theorem C that the only options for such a fusion system are either a unique exotic system on a group of order $7^5$, or the fusion system of $\sp_4(p)$ on its Sylow $p$-subgroup. The group of order $7^5$ hosting the exotic fusion system is a maximal subgroup of a Sylow $7$-subgroup of the Monster group. Here we prove that the fusion system is block-exotic too, giving rise to our main theorem:

\begin{Theorem}
\label{main}
\textit{Let $p \geq 5$ be a prime, $P$ a $p$-group of sectional rank $3$ and $\mathcal F$ be a fusion system on $P$ with $\mathcal O_p(\mathcal F)=1$. Then $\mathcal F$ is exotic if and only if it is block-exotic}.
\end{Theorem}

See \cite{ako} for details on (block) fusion systems or Section 2 of \cite{S20} for a more compact overview of the terms needed. In the only remaining section we prove our main theorem, using the Classification of Finite Simple Groups.

\section{Proof of Main Theorem}
Fix $P$ to be a maximal subgroup of a Sylow $7$-subgroup of the monster. We prove that the exotic fusion system on $P$ described by Grazian in \cite{grazian} is block-exotic too. We first reduce the problem to quasisimple groups and state the reduction theorem we apply. Recall that a fusion system is called reduction simple, when it has no non-trivial proper strongly $\mathcal F$-closed subgroups, i.e. subgroups $Q$ such that $\varphi(R) \leq Q$ for all $\varphi \in \mor(\mathcal F)$ defined on $R \leq Q$.

\begin{Theorem}
\cite[Theorem 3.5]{S20}
\label{glastonbury}
\textit{Let $P$ be a non-abelian $p$-group such that $Z(P)$ is cyclic and let $\mathcal F$ be a reduction simple fusion system on $P$. If $\mathcal F$ is block-realizable, then there exists a fusion system $\mathcal F_0$ on $P$ and a quasisimple group $L$ with an $\mathcal F_0$-block, where $\mathcal O_p(\mathcal F_0)=1$}.
\end{Theorem}

\begin{Proposition}
\label{redu}
\textit{Let $\mathcal F$ be the exotic fusion system on $P$. If $\mathcal F$ is block-realizable, it is block-realizable by the block of a finite quasisimple group}.
\end{Proposition}

\begin{proof} Clearly $P$ has cyclic centre. Assume $1 \neq N \leq P$ is strongly $\mathcal F$-closed. In particular $N \unlhd P$, which implies $Z(P) \leq N$. Thus, as in the proof of \cite[Theorem 4.3.1]{grazianpearls}, we obtain $N=P$, which means that $\mathcal F$ is reduction simple. By Theorem \ref{glastonbury}, there exists a fusion system $\mathcal F_0$ on $P$ and a quasisimple group $L$ with an $\mathcal F_0$-block, where $\mathcal O_p(\mathcal F_0)=1$. However, by Theorem C of \cite{grazian} we deduce $\mathcal F=\mathcal F_0$.
\end{proof}

The following result is useful for proving block-exoticity.

\begin{Proposition}
\label{bologna}
\cite[Proposition 4.3]{S20}
\textit{Let $G$ be a quasisimple finite group and denote the quotient $G/Z(G)$ by $\overline G$. Suppose $\overline{G}=G(q)$ is a finite group of Lie type and let $p$ be a prime number $\geq 7$, $(p,q)=1$. Let $D$ be a $p$-group such that $Z(D)$ is cyclic of order $p$ and $Z(D) \subseteq [D,D]$. If $D$ is a defect group of a block of $G$, then there are $n, k \in \mathbb{N}$ and a finite group $H$ with $\slg_n(q^k) \leq H \leq \gl_n(q^k)$ $($or $\su_n(q^k) \leq H \leq \gu_n(q^k))$ such that there is a block $c$ of $H$ with non-abelian defect group $D'$ such that $D'/Z$ is of order $|D/Z(D)|$ for some $Z \leq D' \cap Z(H)$}.
\end{Proposition}

\begin{Proposition}
\label{hjelp}
\textit{If $G$ is as in the previous proposition, then $G$ has no blocks with defect groups isomorphic to $P$}.
\end{Proposition}

\begin{proof} Recall $|P|=7^5$ and $Z(P) \cong C_7$. We apply Proposition \ref{bologna} with $P$ taking the role of $D$. Let $H$, $D'$ be as in its assertion with $p=7$. Assume first $H \leq \gl_n(q^k)$ and let $a$ be such that $|q^k-1|_7=7^a$. Then, since $\slg_n(q^k) \leq H$, we have $|D'|=|P/Z(P)| \cdot |Z|=7^4|Z| \leq 7^4|Z(H)| \leq 7^4|Z(\slg_n(q^k))|\leq 7^{4+a}$. Now the block of $k\gl_n(q^k)$ covering $c$ has a defect group of order at most $7^{2a+4}$. But it is a well-known fact, that (non-abelian) defect groups of $\gl_n(q^k)$ have order at least $7^{7a+1}$, see \cite[Theorem 3C]{torquay}. Thus, $7^{7a+1} \leq 7^{2a+4}$, which is a contradiction. The case $H \leq \gu_n(q^k)$ can be shown in the same fashion by considering the $7$-part of $q^k+1$ instead of $q^k-1$. \end{proof}

\textit{Proof of Theorem \ref{main}}.
By \cite[Theorem C]{grazian}, the only exotic fusion system on such a group is a fusion system on a maximal subgroup $P$ of a Sylow $7$-subgroup of the Monster with $|P|=7^5$. Thus, by \cite[Theorem 3.6]{radha}, we only need to prove block-exoticity of this system, denote it by $\mathcal F$.\\
By Proposition \ref{redu}, if $\mathcal F$ is block-realizable, we may assume that it is block-realizable by the block of a finite quasisimple group $G$ having $P$ as defect group. We use the classification of finite simple groups to exclude all possibilites for $G$. Clearly, we can assume that $G$ is non-abelian.\\
Firstly, assume $G/Z(G)$ is an alternating group $\mathfrak A_m$. Then $P$ is isomorphic to a Sylow $7$-subgroup of some symmetric group $\mathfrak S_{7w}$ with $w \leq 6$. Define the cycle $\sigma_i=((i-1)7+1,(i-1)7+2,\dots,i7)$ and the subgroup $S'=\langle \sigma_1, \dots, \sigma_6 \rangle\leq \mathfrak A_m$. Then $S' \in \syl_7(\mathfrak A_m)$. But this group is abelian, which means that $P \notin \Syl_7(\mathfrak A_m)$ and thus it cannot be the defect group of a block of $G$.\\
Next, assume $G$ is a group of Lie type. First assume the latter group is defined over a field of characteristic $7$, then $P$ is a Sylow $7$-subgroup of $G$ by \cite[Theorem 6.18]{caen} and thus $\mathcal F$ cannot be exotic. In particular, we can assume $G$ is defined over a field of order coprime to $7$. However, by Proposition \ref{hjelp}, we see that case in cross characteristic is not possible either.\\
Finally, assume $G/Z(G)$ is one of the sporadic groups. By, \cite[Theorem 9.22]{david}, the fusion system of a block of a sporadic group can not be exotic. This proves the theorem. \hfill$\square$\\

\bigskip

{\it Acknowledgements.} I thank Chris Parker and the University of Halle--Wittenberg to provide a platform to meet him. Furthermore, I thank the referee for carefully reading my manuscript and Pham Huu Tiep for communicating this paper.


\end{document}